\documentclass[11pt]{amsart}
\usepackage{hyperref}
\usepackage{amsfonts}
\usepackage{amsmath}
\usepackage{amssymb}
\usepackage{amscd}
\usepackage{graphicx}
\usepackage{enumitem}
\usepackage{latexsym}
\usepackage{color}
\usepackage{comment}
\usepackage{epsfig}
\usepackage{amsopn}
\usepackage{xypic}
\usepackage{mathrsfs}
\usepackage{array}
\usepackage[usenames,dvipsnames]{pstricks}
\usepackage{epsfig}
\usepackage{pst-grad} % For gradients
\usepackage{pst-plot} % For axes
\usepackage[space]{grffile} % For spaces in paths
\usepackage{etoolbox} % For spaces in paths
\makeatletter % For spaces in paths
\patchcmd\Gread@eps{\@inputcheck#1 }{\@inputcheck"#1"\relax}{}{}
\makeatother

\hypersetup{pdfpagemode=UseNone}
\usepackage{booktabs}
\usepackage{longtable}
\theoremstyle{plain}
\newtheorem{lemma}{Lemma}[section]
\newtheorem*{theorem*}{Theorem}
\newtheorem*{lemma*}{Lemma}
\newtheorem*{proposition*}{Proposition}
\newtheorem*{conjecture*}{Conjecture}
\newtheorem*{corollary*}{Corollary}
\newtheorem*{problem*}{Problem}
\newtheorem{theorem}[lemma]{Theorem}

\newtheorem{corollary}[lemma]{Corollary}
\newtheorem{proposition}[lemma]{Proposition}

\theoremstyle{definition}

\newtheorem{example}[lemma]{Example}
\newtheorem{remark}[lemma]{Remark}

\newcommand{\CC}{\mathbb{C}}

\newcommand{\OO}{\mathcal{O}}
\newcommand{\te}{\otimes}

\newcommand{\cF}{\mathcal F}

\newcommand{\bv}{{\bf v}}

\newcommand{\cE}{\mathcal{E}}

\newcommand{\ZZ}{\mathbb{Z}}

\renewcommand{\cL}{\mathcal{L}}
\newcommand{\PP}{\mathbb{P}}

\DeclareMathOperator{\Pic}{Pic}

\DeclareMathOperator{\rk}{rk}

\DeclareMathOperator{\ext}{ext}

\DeclareMathOperator{\sHom}{\mathcal{H}\kern -.5pt\mathit{om}}
\DeclareMathOperator{\sTor}{\mathcal{T}\kern -1.5pt\mathit{or}}

\DeclareMathOperator{\PGL}{PGL}

\newcommand{\leqor}{\underset{{\scriptscriptstyle (}-{\scriptscriptstyle )}}{<}}

\begin{document}

\date{\today}
\author[I. Coskun]{Izzet Coskun}
\address{Department of Mathematics, Statistics and CS \\University of Illinois at Chicago, Chicago, IL 60607}
\email{coskun@math.uic.edu}

\author[J. Huizenga]{Jack Huizenga}
\address{Department of Mathematics, The Pennsylvania State University, University Park, PA 16802}
\email{huizenga@psu.edu}

\author[J. Kopper]{John Kopper}
\address{Department of Mathematics, The Pennsylvania State University, University Park, PA 16802}
\email{kopper@psu.edu}

\thanks{During the preparation of this article the first author was partially supported by the NSF FRG grant DMS 1664296
and the second author was partially supported by NSF FRG grant DMS 1664303.}

\title{Disconnected moduli spaces of stable bundles on surfaces}

\begin{abstract}
  We use hypersurfaces containing unexpected linear spaces to construct interesting vector bundles on complete intersection surfaces in projective space. We discover examples of moduli spaces of rank 2 stable bundles on surfaces of Picard rank one with arbitrarily many connected components. \end{abstract}

\maketitle

\setcounter{tocdepth}{1}
\tableofcontents

\section{Introduction}
In this note we construct  examples of moduli spaces of rank two bundles on surfaces  of Picard rank one with arbitrarily many connected components.

On a smooth projective curve, the moduli space of semistable bundles of rank $r$ and degree $d$ is irreducible \cite[Cor. 4.5.5]{HuybrechtsLehn}. In contrast, if $\dim(X) \geq 3$, then even the Hilbert scheme of points on $X$ is reducible with components of different dimensions \cite{Iarrobino}. There are many constructions of reducible and disconnected moduli spaces if $\dim(X) \geq 3$ (for example, see \cite{ein}). Moduli spaces of sheaves on surfaces exhibit the most interesting behavior. The Hilbert scheme of $n$ points on a smooth, irreducible projective surface  is smooth and irreducible \cite{Fogarty}. In higher rank,  the philosophy of Donaldson \cite{Donaldson} and Gieseker and Li  \cite{GiesekerLi} expects the moduli spaces $M_X(r, c_1, c_2)$ to become better behaved as the second Chern class $c_2$ tends to infinity. For example, they become reduced and irreducible of the expected dimension if $c_2$ is sufficiently large \cite{OGrady}. On the other hand, if $c_2$ is small, then $M_X(r, c_1, c_2)$ may be reducible with components of different dimensions and may have nonreduced components \cite{Mestrano, MestranoSimpson}. For example, \cite{CoskunHuizengaPathologies} and \cite{Friedman2} construct moduli spaces with arbitrarily many irreducible components. It is reasonable to expect that for small $c_2$, the moduli spaces of sheaves on surfaces satisfy a version of Murphy's Law \cite{Vakil}.

One can construct disconnected moduli spaces of sheaves on  threefolds  by using the Serre correspondence, which relates codimension 2 subschemes  to rank 2 vector bundles (see \S\ref{sec:preliminaries} or \cite{Starr}). Subschemes with a given Hilbert polynomial are parameterized by a Hilbert scheme. If this Hilbert scheme is disconnected, the same can be true of the moduli space of sheaves associated through the Serre correspondence. We use a variant of this argument by relating the moduli space $M_X(r, c_1, c_2)$ on the surface $X$ to a disconnected Hilbert scheme of linear spaces on an ambient higher-dimensional variety. 

Okonek and Van de Ven \cite{OkonekVandeven} and Kotschick \cite{Kotschick} found examples of disconnected moduli spaces on elliptic surfaces in their study of the topology of the underlying real fourfold (see also \cite{Friedman}). These disconnected moduli spaces are on certain elliptic surfaces with large Picard rank. In contrast, our construction uses general type surfaces with Picard rank one.

Let $n \geq 4$ be an integer.  Let $2 < d_1 < d_2 \leq  \cdots \leq d_{n-2}$ be integers. If $n=4$, further assume that $d_1 \geq 4$ and $d_2 \geq 6$. Set $e = \prod_{i=2}^{n-2} d_i.$  Let $D_1 \subset \PP^n$ be a hypersurface of degree $d_1$ that contains a linear space $\Phi$ of dimension $n-3$.  Assume that the singular locus of $D_1$ has codimension  5. For $2\leq i \leq n-2$, let $D_i \subset \PP^n$ be very general hypersurfaces of degree $d_i$. Let $X$ be the  complete intersection  $D_1 \cap \cdots \cap D_{n-2}$, which  is a smooth, projective surface of degree $d_1 e$. Let $H$ denote the hyperplane class on $X$. Let $F_{n-3}(D_1)$ denote the Fano scheme parameterizing linear spaces on $D_1$ of dimension $n-3$.

\begin{theorem*}[\protect{\ref{thm:main_thm}}]
Every connected component of $F_{n-3}(D_1)$ corresponds to a distinct connected component of $M_X(2,H,(d_1-1)e)$ of  the same dimension.
\end{theorem*}

In \S \ref{sec:examples}, we give two families of examples of $D_1 \subset \PP^4$. In Example \ref{ex:many_components}, $F_1(D_1)$ has $3d_1^2$ zero-dimensional connected components.   In Example \ref{ex:many_components3}, $F_1(D_1)$ has $d_1$ one-dimensional connected components. Example \ref{ex:many_components2} gives examples of $D_1 \subset \PP^5$ where $F_2(D_1)$ has  $15 d_1^3$ zero-dimensional connected components. We conclude the following.

\begin{corollary}
For any integer $k$, there exists a smooth surface $X$ and Chern character $\bv$ on $X$ such that $M_X(\bv)$ has at least $k$ connected components.
\end{corollary}

The original motivation for this note was a question about monotonicity of Betti numbers of $M_X(r,c_1,c_2)$. Coskun and Woolf \cite{CoskunWoolf} conjecture that as $c_2$ tends to infinity, the Betti numbers of $M_X(r, c_1, c_2)$ stabilize.  G\"{o}ttsche proved that the  Betti numbers of the Hilbert scheme of $n$ points $X^{[n]}$ stabilize as $n$ tends to infinity \cite{Gottsche}. In fact, the Betti numbers monotonically increase to the stable value as $n$ increases. The same monotonicity occurs for certain moduli spaces of higher rank sheaves on K3 surfaces and rational surfaces. In talks given by the first author, several mathematicians raised the question whether one should expect the Betti numbers to always  monotonically increase to the stable value. As examples of disconnected moduli spaces show, this is already false for $b_0$.

\begin{corollary}
The Betti numbers of the moduli spaces of sheaves $M_X(r, c_1, c_2)$  do not monotonically increase as $c_2$ increases.
\end{corollary}

One could refine the question and ask whether the Betti numbers monotonically increase once $c_2$ reaches the O'Grady bound and the moduli spaces become reduced and  irreducible  of the expected dimension.  

\subsection*{Organization of the paper} In \S \ref{sec:preliminaries}, we recall basic facts about moduli spaces and complete intersections. In \S \ref{sec:complete_intersections}, we describe our main construction and prove Theorem \ref{thm:main_thm}. In \S \ref{sec:examples}, we give examples that illustrate interesting phenomena.

\subsection*{Acknowledgments} We would like to thank Lawrence Ein, Joe Harris, Daniel Huybrechts,  Julius Ross, Dennis Tseng, Matthew Woolf and Kota Yoshioka for valuable conversations. We thank the organizers and participants of ZAG, Rahul Pandharipande, and the participants of the Algebraic Geometry and Moduli seminar for stimulating discussions that led to this note.

\section{Preliminaries}\label{sec:preliminaries}
In this section, we recall basic facts concerning moduli spaces of sheaves (see  \cite{HuybrechtsLehn, MestranoSimpson2} for further details) and complete intersections. 

\subsection*{Moduli spaces of sheaves} Let $(X,H)$ be a smooth, polarized, complex projective variety of dimension $n$. If $\cF$ is a torsion-free coherent sheaf on $X$ of rank $\rk(\cF)$, then the Hilbert polynomial $P_{\cF}(m)$ and the reduced Hilbert polynomial $p_{\cF}(m)$ of $\cF$ are defined  by
\[
P_{\cF}(m) = \chi(\cF(mH)) = \rk(\cF) \frac{m^n}{n!} + \text{l.o.t.} \quad \mbox{and} \quad p_{\cF}(m) = \frac{P_{\cF}(m)}{\rk(\cF)}.
\]
The sheaf $\cF$ is called \emph{(semi)stable} if for every proper subsheaf $\cF' \subset \cF$, we have $p_{\cF'}(m) \leqor p_{\cF}(m)$ for $m \gg 0$. Gieseker and Maruyama constructed projective moduli spaces $M_X(\bv)$ parameterizing S-equivalence classes of semistable sheaves on $X$ with Chern character $\bv$ \cite{Gieseker,Maruyama}. The \emph{$H$-slope} of $\cF$ is defined to be the number
\[
\mu_H(\cF) = \frac{c_1(\cF)\cdot H^{n-1}}{\rk(\cF)H^n}.
\]
If $\mu_H(\cF') < \mu_H(\cF)$ for all proper subsheaves $\cF' \subset \cF$, then $\cF$ is called {\em $\mu_H$-stable}.

\subsection*{The Serre correspondence} Let $X$ be a smooth projective surface. Let $Z\subset X$ be a zero-dimensional Gorenstein subscheme of length $n$. Then $Z$ satisfies the {\em Cayley-Bacharach property} for a line bundle $\mathcal{L}$ on $X$ if for any subscheme $W \subset Z$ of length $n-1$, any section of $\mathcal{L}$ vanishing on $W$ vanishes on all of $Z$. Given a line bundle $\mathcal{L}$ and a zero-dimensional scheme $Z$ satisfying the Cayley-Bacharach property for $\omega_X\te \cL$, the Serre correspondence constructs a  vector bundle of rank $2$ on $X$.

\begin{theorem}[The Serre Correspondence \protect{\cite[5.1.1]{HuybrechtsLehn}}]
  Let $X$ be a smooth projective surface, and $Z \subset X$  be a local complete intersection subscheme of dimension $0$ and length $n$. Let $\mathcal{L}$ be a line bundle on $X$. Then there exists an extension
  \[
    0 \to \OO_X \to \mathcal{E} \to \mathcal{L}\otimes I_Z \to 0
  \]
  with $\mathcal{E}$ locally free if and only if $Z$ satisfies the Cayley-Bacharach property for $\omega_X \otimes \mathcal{L}$.
  \end{theorem}

\subsection*{The Cayley-Bacharach Theorem} For our construction, we will use the following  Cayley-Bacharach theorem for projective space  (see \cite[Theorem CB7]{EisenbudGreenHarris})\footnote{The statement of Theorem CB7 in \emph{op. cit.} should say divisors instead of curves.}.
\begin{theorem}[Cayley-Bacharach]\label{thm:cayley_bacharach}
  Let $D_1, D_2, \dots, D_n \subset \PP^n$ be divisors of degrees $d_1, d_2, \dots, d_n$ respectively, meeting in a zero-dimensional scheme $Z = D_1 \cap D_2 \cap \cdots \cap D_n$. Set $d=\sum_{i=1}^{n}d_i$.
  \begin{enumerate}
  \item If $D \subset \PP^n$ is any divisor of degree $d - n -1$ containing a subscheme $Z'$ of $Z$ of length one less than that of  $Z$, then $D$ contains all of $Z$.
  \item Suppose $Z',Z''$ are residual subschemes of $Z$. Then for any $m \leq d - n-1$, we have
    \[
h^0(\PP^n, I_{Z'}(m))-h^0(\PP^n, I_Z(m)) = h^1\left(\PP^n,I_{Z''}\left(d-n-1-m\right)\right).
      \]
        \end{enumerate}

\end{theorem}

\subsection*{Complete intersections} Finally, we will use the following standard facts about complete intersections and their Fano schemes of linear spaces. Given a variety $X$, let $F_j(X)$ denote the Fano scheme parameterizing $j$-dimensional linear spaces contained in $X$.

\begin{proposition}\label{prop:Fano}
\begin{enumerate}
\item If $D \subset \PP^4$ is a smooth hypersurface of degree $d \geq 4$, then $\dim(F_1(D)) \leq 1$. If $d \geq 6$ is and $D$ is general, then $F_1(D)$ is empty. 
\item If $D \subset \PP^5$ is a smooth hypersurface of degree $d \geq 3$, then $\dim(F_2(D)) \leq 0$. If $D$ is general, then $F_2(D)$ is empty.
\item Let $n \geq 5$. If $D \subset \PP^n$ is a hypersurface of degree $d \geq 3$ which is smooth in codimension 4, then $\dim(F_{n-3}(D)) \leq 0$.
\end{enumerate}
\end{proposition}

\begin{proof}
Let $D \subset \PP^4$ be a smooth hypersurface of degree $d$. Let $Y$ be a component of $F_1(D)$ containing a line $\ell$. If $d \geq 4$, then $\deg(N_{\ell/D}) = 3-d<0$. Hence, $\ell$ is not free and the lines parameterized by $Y$ can only sweep a surface. If $\dim(Y)\geq 2$, this surface must be a plane. However, by the Lefschetz hyperplane theorem, a smooth hypersurface of degree $d>1$ cannot contain a plane. We conclude that $\dim(F_1(D)) \leq 1$. A  dimension count shows that a general hypersurface in $\PP^n$ of degree $d \geq 2n-2$ does not contain lines. This proves (1).

Let $D \subset \PP^5$ be a smooth hypersurface of degree $d\geq 3$ that contains a plane $\Lambda$. Choosing coordinates, we may assume that $\Lambda$ is defined by $z_0=z_1=z_2=0$ and that the equation of the hypersurface is given by $z_0 F_0 + z_1 F_1 + z_2 F_2=0,$ where $F_i$ are polynomials of degree $d-1$.  For $1 \leq i \leq 3$, let $\overline{F}_i$ denote the restriction of  $F_i$ to $\Lambda$. If the $\overline{F}_i$  have a common zero on $\Lambda$, then $D$ is singular at that point. The standard normal bundle sequence for $N_{\Lambda/D}$ is given by $$0 \to N_{\Lambda/D} \to \OO_{\Lambda}(1)^3 \stackrel{(\overline{F_0},\overline{F_1},\overline{F_2})}{\longrightarrow} \OO_{\Lambda}(d) \to 0.$$ If $ N_{\Lambda/D}$ has a section, then the $\overline{F}_i$ must satisfy a relation of the form $\sum L_i \overline{F_i} =0$ for some linear forms $L_i$. Since the polynomials $\overline{F_i}$ have degree at least $2$, they must have a common solution, contradicting the smoothness of $D$. Therefore $H^0(\Lambda,  N_{\Lambda/D})=0$ and $\dim(F_2(D)) \leq 0$. Finally, a dimension count shows that for general $D$, $F_2(D)$ is empty. This proves (2), and (3) is an immediate consequence of (2) by intersecting with a general $\PP^5$. 
\end{proof}

\begin{proposition}\label{prop:no_positive_intersections}
Let $4 \leq d_1 \leq \cdots \leq d_{n-3}$ be a sequence of integers such that $\sum_{i=1}^{n-3} d_i > n+1$. Let $X \subset \PP^n$ be a very general complete intersection of type $d_1, \dots, d_{n-3}$. Then any codimension 3 linear space intersects $X$ in dimension 0.
\end{proposition}

\begin{proof}
When $n=4$ or $5$, the proposition asserts that $X$ does not contain a line or a plane curve, respectively. A simple dimension count shows that this is the case as soon as $\sum_{i=1}^{n-3} d_i > n+1$. When $n>5$, it suffices to show that the codimension of the locus of forms $(D_1, \dots, D_{n-3})$ that contain a curve in $\PP^{n-3}$ is greater than the dimension of the Grassmannian $\mathbb{G}(n-3,n)$.  This  follows from the dimension counts in  \cite[Lemmas 4.2 and 4.3]{Dennis}.
\end{proof}

\section{The construction}\label{sec:complete_intersections} 
In this section, we will describe our construction and prove its main properties. We first establish some notation.

\subsection*{The surfaces}Let $n \geq 4$ be an integer.  Let $2 < d_1 < d_2 \leq  \cdots \leq d_{n-2}$ be integers such that $d = \sum_{i=1}^{n-2} d_i \geq n+1$. Set $e = \prod_{i=2}^{n-2} d_i.$  Let $D_1 \subset \PP^n$ be a hypersurface of degree $d_1$ that contains a linear space $\Phi$ of dimension $n-3$.  When $4 \leq n \leq 5$, we will assume that $D_1$ is smooth. If $n>5$, then $D_1$ cannot be smooth, but we will assume that its singular locus has codimension 5. For $2\leq i \leq n-2$, let $D_i \subset \PP^n$ be very general hypersurfaces of degree $d_i$. Let $X$ be the  complete intersection  $D_1 \cap \cdots \cap D_{n-2}$, which  is a smooth, projective surface of degree $d_1 e$. Let $H$ denote the hyperplane class on $X$. The following proposition summarizes the basic properties of $X$.

\begin{proposition}\label{prop:basic_properties_surface}
The surface $X$ satisfies the following properties.
   \begin{enumerate}
  \item The canonical class of $X$ is $K_X = (d-n-1)H$.
  \item We have $\Pic X = \ZZ H$.
  \item For any integer $m$ and $i=1,2$,  $h^i(\PP^n,I_X(m)) = 0$.
    \item For any integer $m$,  $h^1(X,\OO_X(m)) = 0$ and the map $H^0(\PP^n, \OO_{\PP^n}(m)) \to H^0(X, \OO_X(m))$ is surjective.
  \end{enumerate}
\end{proposition}
\begin{proof}
  The canonical class of $X$ is computed by adjunction. The intersection $Y= \cap_{i=1}^{n-3} D_i$ is a smooth, complete intersection threefold. By the Lefschetz hyperplane theorem, $\Pic(Y)$ is cyclic generated by the restriction of the hyperplane class. By assumption $D_{n-2}$ is very general and $K_X$ has a section, so Moishezon's Noether-Lefschetz theorem \cite{Moishezon} implies that $\Pic(X) \cong \ZZ H$. 
  
  Set $W = \oplus_{i=1}^{n-2} \OO_{\PP^n}(-d_i)$. Since $X$ is a complete intersection, the Koszul complex
  $$0 \to \bigwedge^{n-2}W \to \bigwedge^{n-3} W \to \cdots \to \bigwedge^2 W \to W \to I_X \to 0$$ gives a resolution of $I_X$. Since $H^i(\PP^n, (\wedge^j W)(m)) =0$ for $0<i<n$ and all $m$, we conclude that $h^i(\PP^n,I_X(m)) = 0$ if $i=1$ or $2$ by an easy diagram chase. Finally, the long exact sequence associated to the standard sequence 
  $$0 \to I_X(m) \to \OO_{\PP^n}(m) \to \OO_X(m)\to 0$$ shows that $h^1(X,\OO_X(m)) = 0$ and the map $H^0(\PP^n, \OO_{\PP^n}(m)) \to H^0(X, \OO_X(m))$ is surjective.  \end{proof}

 \subsection*{The bundles} By assumption, $D_1$ contains a linear space $\Phi \cong \PP^{n-3}.$ Let $\Lambda \cong \PP^{n-2} \subset \PP^n$ be a linear subspace of codimension 2 containing $\Phi$. Since the singularities of $D_1$ occur in codimension 5, we must have $\Lambda \not\subset D_1$ and $\Lambda \cap D_1=\Phi \cup Y$, where $Y$ is a degree $d_1 - 1$ hypersurface in $\Lambda$. Furthermore, since $D_2, \dots, D_{n-2}$ are very general hypersurfaces, $Z= Y \cap D_2 \cap \cdots \cap D_{n-2}$ is a zero-dimensional complete intersection scheme of length $(d_1-1)e$. Moreover, if $\Lambda$ is general, then $Z$ is reduced. 

\begin{proposition}\label{prop:CB_for_X}
The scheme $Z\subset X$ satisfies the Cayley-Bacharach property for the line bundle $\OO_X(d-n) \cong \omega_X (1)$.
\end{proposition}
\begin{proof}
Observe that the scheme $Z$ is a complete intersection of degrees $d_1-1,d_2, \dots, d_{n-2}, 1, 1$ in $\PP^n$. Hence,  by Theorem \ref{thm:cayley_bacharach}, the scheme $Z$ satisfies the Cayley-Bacharach property for $\OO_{\PP^n}(d-n)$. By Proposition \ref{prop:basic_properties_surface}, every section of $\OO_X(d-n)$ is a restriction of a section of $\OO_{\PP^n}(d-n)$. Hence, $Z$ satisfies the Cayley-Bacharach property for $\OO_X(d-n)$.
\end{proof}

By the Serre correspondence, there is a locally free sheaf $\cE$ on $X$ defined by a sequence
\begin{equation}\label{eqn:defining_seq_E}\tag{$\ast$}
  0 \to \OO_X \to \mathcal{E} \to I_Z(1) \to 0.
\end{equation}
The following proposition summarizes the basic properties of $\cE$.

\begin{proposition}\label{prop:basic_properties_bundle}
Let $\cE$ be the locally free sheaf defined by  (\ref{eqn:defining_seq_E}). Then:
\begin{enumerate}
\item We have $\ext^1(I_Z(1),\OO_X) = 1$, so up to scalars there is a unique non-split extension $\cE$.
\item The bundle $\cE$ is $\mu_H$-stable and $\cE \in M_X(2, H, (d_1-1)e)$.
\item Let $\delta_{i,j}$ denote the Kronecker delta function. Then $h^0(X, \cE) = 3 + \delta_{2, d_1}$ and $h^1(X, \cE)=0$.
\end{enumerate}
\end{proposition}
\begin{proof}
Throughout the proof, let $I_Z$ denote the ideal sheaf of $Z$ in $X$. We will write $I_{Z/Y}$, such as $I_{Z/\PP^n}$, when referring to the ideal sheaf of $Z$ in an ambient space $Y$. 

By Serre duality, we have
  \[
\ext^1(I_Z(1),\OO_X) = \ext^1(\OO_X,I_Z(K_X+H)) = h^1(X,I_Z(d-n)).
\]
By the Cayley-Bacharach theorem, $h^1(\PP^n, I_{Z/\PP^n}(d-n))=1$. By Proposition \ref{prop:basic_properties_surface}, $$H^1(\PP^n,  I_X(d-n)) = H^2(\PP^n,  I_X(d-n))=0.$$
The long exact sequence of cohomology associated to
  \[
    0 \to I_X(d-n) \to I_{Z/\PP^n} (d-n)\to I_{Z} (d-n)\to 0.
  \]
now implies that $h^1(X, I_Z(d-n))= \ext^1(I_Z(1),\OO_X) =1$, proving  (1).

Since $\Pic(X)\cong \ZZ H$, if $\cE$ is not $\mu_H$-stable, then there exists a subbundle $\OO_X(m) \subset \cE$ with $m \geq 1.$ The composition
$\OO_X(m) \to \cE \to I_X(1)$ must be the zero map, so $\OO_X(m)$ must factor through the map $\OO_X \to \cE$ in the sequence (\ref{eqn:defining_seq_E}). This is impossible if $m \geq 1$, so $\cE$ must be $\mu_H$-stable. A simple Chern class computation shows that $\cE \in M_X(2, H, (d_1-1)e)$. This proves (2).

If $d_1 > 2$, then $Z$ spans $\Lambda$. Hence, the hyperplanes containing  $Z$ are the hyperplanes containing $\Lambda$  and $h^0(X,I_Z(1))=2$. Since $h^1(X,\OO_X) = 0$, the long exact sequence associated to (\ref{eqn:defining_seq_E}) shows that $h^0(X,\cE) = 1 + h^0(X,I_Z(1)) = 3.$ If $d_1=2$, then $Z$ spans a linear space of dimension $n-3$ and $h^0(X, I_Z(1))=3$. In this case, we get $h^0(X,\cE) = 1 + h^0(X,I_Z(1)) = 4.$ 

In order to calculate $h^1(X, \cE)$, we first calculate $h^2(X,\cE) = h^0(X,\cE^*(d-n-1))$. Since $\cE$ has rank 2, we have the isomorphism $$\cE^*\cong  \cE \otimes (\det \cE)^* = \cE(-1).$$ Thus
  \[
h^2(X,\cE) = h^0(X,\cE(d-n-2)).
\]
We twist the sequence (\ref{eqn:defining_seq_E}) by $\OO_X(d-n-2)$ to get
\[
  0 \to \OO_X(d-n-2) \to \cE(d-n-2) \to I_Z(d-n-1) \to 0.
\]
Since $h^1(X, \OO_X(d-n-2))=0$, we obtain 
\begin{equation}\label{eqn:h0_E(d-n-1)}
  h^0(X,\cE(d-n-2)) = h^0(X,\OO_X(d-n-2)) + h^0(X,I_Z(d-n-1)).
\end{equation}
By Serre duality and the fact that the equations defining $X$ have degree at least 2, we have $$h^2(X, \OO_X(d-n-2))= h^0(X, \OO_X(1))=n+1.$$
 Hence,
\begin{equation}\label{eqn:h0_OX(d-n-2)}
  h^0(X,\OO_X(d-n-2)) =\chi(\OO_X(d-n-2)) - (n+1).
\end{equation}
Consider the ideal sheaf exact sequence
 \[ 0 \to I_X(d-n-1) \to I_{Z/\PP^n} (d-n-1)\to I_{Z}(d-n-1) \to 0.
  \]
By Proposition \ref{prop:basic_properties_surface},  $h^1(\PP^n,I_X(d-n-1)) = 0$ and 
\begin{equation}\label{eqn:I_Z}
    h^0(X,I_Z(d-n-1)) = h^0(\PP^n,I_{Z/\PP^n}(d-n-1)) - h^0(\PP^n,I_X(d-n-1)).
\end{equation}
The value of $h^0(\PP^n,I_X(d-n-1))$ is easily determined from the exact sequence
\[
0 \to I_X(d-n-1) \to \OO_{\PP^n}(d-n-1) \to \OO_X(d-n-1) \to 0.
\]
Since $\OO_X(d-n-1) = \omega_X$, we have $h^1(X, \omega_X)=0$ and $h^2(X,\omega_X) = 1$. Furthermore, $h^3(\PP^n, I_X(d-n-1))=1$ and $h^i(\PP^n, I_X(d-n-1))=0$ for $i \neq 0,3$. Hence, 
\begin{equation}\label{eqn:h0_IX(d-n-1)}
h^0(\PP^n,I_X(d-n-1)) = \chi(I_X(d-n-1))+1.
\end{equation}
The scheme $Z$ is a complete intersection of type $1,1,(d_1-1), d_2, \dots, d_{n-2}$.   The Koszul complex for $I_{Z/\PP^n}(d-n-1)$  implies that
$h^1(\PP^n, I_{Z/\PP^n}(d-n-1))= n-1 - \delta_{2, d_1}$. We conclude that 
$$h^0(\PP^n, I_{Z/\PP^n}(d-n-1)) = \chi( I_{Z/\PP^n}(d-n-1)) + n-1-  \delta_{2, d_1}.$$
Combining this with  Equations (\ref{eqn:h0_E(d-n-1)}), (\ref{eqn:h0_OX(d-n-2)}), (\ref{eqn:I_Z}), and (\ref{eqn:h0_IX(d-n-1)}), we obtain
\[
h^2(\cE)= h^0(X,\cE(d-n-2)) = \chi(\OO_X(d-n-2)) + \chi(I_X(d-n-1))- 3 - \delta_{2, d_1}.
\]
Since
$\chi(\cE(d-n-2)) =  \chi(\OO_X(d-n-2)) + \chi(I_X(d-n-1))$ and $h^2(\cE(d-n-2))= h^0(\cE) = 3 + \delta_{2,d_1}$, we conclude that 
 $h^1(X,\cE) = h^1(\cE(n-d-2)) = 0$. This concludes the proof of the proposition.
\end{proof}

We are now ready to state and prove our main theorem. 

\begin{theorem}\label{thm:main_thm}
Let $n \geq 4$ and $d_1 > 2$.  If $n=4$, assume that $d_1 \geq 4$ and $d_2 \geq 6$. Then for every connected (respectively, irreducible) component of the Fano scheme $F_{n-3}(D_1)$, the moduli space $M_X(2,H,(d_1-1)e)$ has a connected (respectively, irreducible) component of the same dimension. 
\end{theorem}

\begin{proof}
Let  $N \subset M_X(2,H,(d_1-1)e)$ be an irreducible component containing the bundle $\cE$ defined by (\ref{eqn:defining_seq_E}). Endow $N$ with its reduced induced structure. By Proposition \ref{prop:basic_properties_bundle}, $h^1(X, \cE)=0$.  Hence, by the upper semicontinuity of cohomology, $h^0(X, \cE') \geq 3$ for every sheaf $\cE' \in N$ and there is an exact sequence
\[
  0 \to \OO_X \to \cE' \to Q \to 0.
\]
Computing Chern classes, we must have $\rk(Q) = 1$, $c_1(Q) = H$, and $c_2(Q) = (1-d_1)e$. We show that $Q$ is torsion-free. Every element $\cE' \in M_X(2, H, (d_1-1)e)$ is $\mu_H$-stable with $c_1(\cE') = H$. Hence, for any integer $m>0$, we must have $h^0(X,\cE'(-m)) = 0$. Since $h^1(X, \OO_X(-m))=0$ by Proposition \ref{prop:basic_properties_surface}, we conclude that $h^0(X,Q(-m)) = 0$, so $Q$ cannot have torsion in dimension zero.

If $Q$ has torsion along a curve in $X$, let $T \subset Q$ be the torsion subsheaf and $Q'$ be the quotient $Q/T$. Then $Q'$ is a quotient of $\cE'$, but $c_1(Q') = aH$ for some $a \leq 0$, which contradicts the stability of $\cE'$. Thus $Q$ is torsion-free, and consequently, $Q = I_{Z'}(1)$ for some scheme $Z'$ of length $(d_1-1)e$. 

The sequence
\[
  0 \to \OO_X \to \cE' \to I_{Z'}(1) \to 0
\]
shows that $h^0(X,I_{Z'}(1)) \geq  2$. We next show that equality must hold. If $$h^0(X,I_{Z'}(1)) >  2,$$ then $Z'$ is contained in a linear space $\Gamma$ of dimension at most $n-3$. The length of $Z'$ is $(d_1-1)e$, which is greater than the degree of $W=\cap_{i=2}^{n-2} D_i$. By B\'ezout's theorem, $W \cap \Gamma$ must contain a curve. Since $D_i$ are very general , $d_i \geq 4$ and $\sum_{i=2}^{n-2} d_i > n+1$, this is not possible by Proposition \ref{prop:no_positive_intersections}. We conclude that   $h^0(X,I_{Z'}(1)) =  2$, $h^0(X,\cE')=3$ and the scheme $Z'$ spans a linear space $\Lambda'$ of dimension $n-2$.

Since $\Pic(X) \cong \ZZ H$, $\Lambda' \cap X$ must be a zero-dimensional complete intersection scheme of degree $d_1 e$ and cannot contain a curve.  Since $\cE'$ is a nontrivial extension of $I_{Z'}(1)$ by $\OO_X$, we have $h^1(X, I_{Z'}(d-n))\geq 1$. The Cayley-Bacharach theorem implies that there exists a linear space $\Phi' \subset \Lambda'$ of dimension at most $n-3$ containing the subscheme $Z'' \subset X \cap \Lambda'$ of length $e$  residual to $Z'$. The dimension of $\Phi'$ must be $n-3$. Otherwise, we could take the span of $\Phi'$ with additional points of $W=\cap_{i=2}^{n-2} D_i$ to obtain a linear space $\Delta$ of dimension $n-3$ which would intersect $W$ in more points than $e= \deg(W)$.  Hence by Bezout's theorem,  $\Delta$ would intersect $W$ in a curve, contradicting Proposition \ref{prop:no_positive_intersections}. We conclude that $h^1(X, I_{Z'}(d-n))= 1$ and $\dim(\Phi')=n-3$.

The scheme $Z''$ of length $e$ is contained in the intersection $\Phi' \cap W$. Since the two schemes have the same length, we conclude that $Z''=\Phi' \cap W$. Since $Z''$ is a complete intersection scheme in $\Phi'$ of hypersurfaces of degree greater than $d_1$, $Z''$ can only be contained in $D_1$ if $\Phi' \subset D_1$.  Reversing the construction, we conclude that $\cE'$ is obtained by the same Serre construction as $\cE$ starting with $\Phi' \in F_{n-3}(D_1)$ instead of $\Phi$. In particular, $H^1(X, \cE')=0$ and $\cE'$ is locally free. 

Let $\mathbb{E}$ be the pushforward of the universal sheaf over $X \times N$ to $N$. By the universal  property of the Fano scheme, we get a morphism $\PP \mathbb{E} \to F_{n-3}(D_1)$. In concrete terms, let  $(\cE', s')$ be a sheaf together with a section up to scaling. Let $Z'$ be the zero locus of $s'$ and $\Lambda'$ its span. The morphism sends $(\cE', s')$  to the span of the points residual to $Z'$ in $\Lambda' \cap X$. Since $\cE'$ and $\cE$ are in the same irreducible component $N$, the $(n-3)$-plane $\Phi'$ must lie in the same irreducible component $K$ of $F_{n-3}(D_1)$ as $\Phi$. In particular, if $n \geq 5$, then Proposition \ref{prop:Fano} gives $\Phi= \Phi'$ and $\cE = \cE'$, and we conclude that $N$ consists of a unique point. The same argument applies if $n=4$ and $\dim(K)=0$. If $n=4$ and $\dim(K)=1$, since there cannot be any nonconstant morphisms from $\PP^2$ to a curve, we conclude that the linear space $\Phi'$ is canonically associated to $\cE'$.  Any irreducible component $N'$ that intersects $N$ must contain a bundle that arises via the Serre correspondence, hence must correspond to an irreducible component $K'$ of $F_{n-3}(D_1)$ that intersects $K$. We thus conclude that the connected component containing $\cE$ corresponds to the connected component of $F_{n-3}(D_1)$ containing $\Phi$ and they have the same dimension.
\end{proof}

\begin{remark}\label{rem:d_1=2}
If $n=4$ or $5$ and $d_1 =2$, then $h^0(X, \cE) =4$ and $h^1(X,\cE)=0$. If $\cE'$ is in the same irreducible component $N$ as $\cE$, then $h^0(X, \cE') \geq 4$. In this case, $Z'$  spans at most a $\PP^{n-3}$. If $n=4$, since $Z'$ has $e>3$ points, it must span a line. If $n=5$ and $Z'$ was contained in a line, then the line would have to be contained in all the hypersurfaces defining $X$ by B\'ezout's theorem. This would contradict the fact that $\Pic(X) \cong \ZZ H$. We conclude that $h^0(X, \cE') = 4$ for every bundle in $N$. The same argument shows that there is a connected component of $M_X(2, H, e)$ for every connected component of $F_{n-3}(D_1)$. In this case, the dimension count  in the proof of Theorem \ref{thm:main_thm} shows that the bundle corresponding to each connected component of $F_{n-3}(D_1)$ is unique.
\end{remark}

\begin{remark}
Assume that $d_1 \geq 3$ when $n\geq 5$ and assume that $d_1 \geq 6$ when $n=4$. Since the general $D_1$ does not contain a linear space of dimension $n-3$, the proof of Theorem \ref{thm:main_thm} shows that the bundles constructed on the special surfaces $X$ do not deform to nearby surfaces. 
\end{remark}

\section{Examples of moduli spaces}\label{sec:examples}
In this section we give explicit examples of  moduli spaces that  exhibit interesting behaviors. We preserve the notation of Section \ref{sec:complete_intersections}.

\begin{example}[Quintic threefold] 
Let $D_1$ be a general quintic threefold in $\PP^4$. Then $D_1$ contains 2875 lines. Let  $D_2$ be a very general hypersurface of degree $d_2 \geq 6$. Let $X = D_1 \cap D_2$. By Theorem \ref{thm:main_thm}, $M_X(2, H, 4d_2)$ contains 2875 zero-dimensional connected components.
\end{example}

\begin{example}[Moduli spaces with arbitrarily many isolated points I]\label{ex:many_components}
Let $d_1 \geq 6$ and let $D_1$ be the hypersurface of degree $d_1$ defined by 
 $$f(z_0, \dots, z_4)= z_0^{d_1} - z_1^{d_1} + z_2^{d_1} - z_3^{d_1} + z_4g(z_0,\dots,z_4)=0$$
where $g$ is a  general homogeneous polynomial of degree $d_1-1$. Then $D_1$ is smooth and contains $3d_1^2$ lines contained in the Fermat surface $f=z_4=0$. Let $\ell$ be one of these lines.  Up to permutations and roots of unity, we may assume that $\ell$ is defined by $z_0-z_1=z_2-z_3=z_4=0$. In the normal bundle exact sequence 
 \[
    0 \to N_{\ell/D_1} \to N_{\ell/\PP^4} \cong  \OO_{\ell}(1)^{\oplus 3} \stackrel{M}{\longrightarrow} \OO_{\ell}(d_1) \to 0,
  \]
the map $M$ is given by $[(d_1-1) s^{d_1-1}, (d_1-1) t^{d_1-1}, g(s,s,t,t,0)]$ in local coordinates $s,t$ on $\PP^1$ \cite[\S 3]{CoskunRiedl}. Hence, for a general choice $g$, $$N_{\ell/D_1} \cong \OO_{\ell}\left(\left\lfloor \frac{3-d_1}{2}\right\rfloor\right) \oplus \OO_{\ell}\left(\left\lceil \frac{3-d_1}{2}\right\rceil\right).$$ This implies $h^0(\ell, N_{\ell/D_1})=0$. We conclude that these $3d_1^2$ lines are isolated points of $F_1(D_1)$. Let $d_2 > d_1$ and let $D_2$ be a very general hypersurface in $\PP^4$ of degree $d_2$. Let $X = D_1 \cap D_2$. By Theorem \ref{thm:main_thm},  $M_X(2, H, (d_1-1)d_2)$ has at least $3d_1^2$ connected components of dimension $0$.
\end{example}

\begin{example}[Spinor bundles]
Let $D_1$ be a smooth quadric fourfold in $\PP^5$. Let $D_2$ and $D_3$ be very general hypersurfaces of degrees $d_2, d_3 \geq 3$, respectively. Let $X = D_1 \cap D_2 \cap D_3$. Since $F_2(D_1)$ has two connected components, by Remark \ref{rem:d_1=2}, $M_X(2, H, d_2d_3)$ has 2 connected components of dimension 0. One can be explicit about these two bundles: they are the restrictions of the two spinor bundles on the quadric $D_1$. Identifying $D_1$ with the Grassmannian $G(2,4)$ under the Pl\"{u}cker embedding, the spinor bundles are the universal quotient bundle $Q$ and the dual $S^*$ of the universal subbundle.  These two bundles restrict to two bundles in $M_X(2, H, d_2d_3)$, which form two connected components. Since $h^0(G(2,4), S^*) = h^0(G(2,4), Q)=4$ and the zero loci of the sections are planes, it is easy to see that the construction in \S \ref{sec:complete_intersections} produces these two bundles.
\end{example}

\begin{example}[Moduli spaces with arbitrarily many isolated points II]\label{ex:many_components2}
Let $d_1 \geq 3$ and let $D_1$ be the Fermat hypersurface in $\PP^5$ defined by 
$\sum_{i=0}^5 x_i^{d_1} =0$. Then $D_1$ contains $15 d_1^3$ planes. Letting $D_2, D_3$ be very general hypersurfaces of degrees $d_2, d_3 > d_1$, Theorem \ref{thm:main_thm} shows that $M_X(2, H, (d_1-1)d_2d_3)$ has at least $15 d_1^3$ connected components. 
\end{example}

\begin{example}[Moduli spaces with arbitrarily many positive dimensional connected components]\label{ex:many_components3}
This example shows that there can be many positive dimensional connected components of $M_X(2,H,(d_1-1)e)$.  Let $g(z_0,z_1)$ and $h(z_2,z_3,z_4)$ be general forms of degree $d_1>5$. Let $D_1 \subset \PP^4$ be defined by the equation  $$f (z_0, \dots, z_4) = g(z_0,z_1)+h(z_2, z_3, z_4)=0.$$ Let $\ell$ be the line defined by $z_2=z_3=z_4=0$ and let $\Lambda$ be the plane defined by $z_0=z_1=0$. Then $\ell$ intersects $D_1$ at the $d_1$ roots of $g(z_0, z_1)$. The tangent hyperplane at one of these points $p_i$ intersects $D_1$ in the cone over the plane curve $z_0=z_1=h(z_2,z_3, z_4)=0$ with vertex at $p_i$. Hence, $D_1$ contains $d_1$ disjoint one-parameter families of lines.

We claim that $D_1$ contains no other lines. Let $U \subset \mathbb{G}(1,4)$ denote the open subset of lines not meeting $\ell$.  Consider the incidence correspondence
\[
I = \{ (g,h,m)\in H^0(\PP^1,\OO_{\PP^1}(d_1)) \times H^0(\PP^2,\OO_{\PP^2}(d_1)) \times U : (g(z_0, z_1)+h(z_2, z_3, z_4))|_m \equiv 0 \}.
\]
The $\PGL_5$ action on $U$ has two orbits: the lines that intersect $\Lambda$ and the lines that do not. We can compute the fiber dimension of $I$ over $U$ by choosing the following representatives of the two orbits
$$m_1: z_0-z_1=z_0-z_3=z_4=0 \quad \mbox{and} \quad m_2: z_0-z_2=z_1 - z_3=z_4=0.$$
Let $g= \sum_{i=0}^{d_1} a_{i}z_0^iz_1^{d_1-i}$ and let $h=\sum_{i,j} b_{i,j} z_2^i z_3^j z_4^{d_1-i-j}$. If $g+h$ vanishes on $m_1$, then the coefficients of $g$ and $h$ must satisfy the $d_1+1$ linear conditions $b_{i, d_1-i} =0$ for $1\leq i \leq d_1$ and $b_{0,d_1} + \sum_{i=0}^{d_1} a_i =0$.  If $g+h$ vanishes on $m_2$, then the coefficients of $g$ and $h$ must satisfy the $d_1+1$ linear conditions $a_i + b_{i, d_1-i}=0$ for $0 \leq i \leq d_1$. Since $d_1+1 > 6 = \dim U$, we conclude that $$\dim I < \dim H^0(\PP^1,\OO_{\PP^1}(d_1)) \times H^0(\PP^2,\OO_{\PP^2}(d_1)).$$  Hence, for a general choice of $g$ and $h$, $D_1$ does not contain a line of $U$. Hence, any line in $D_1$  intersects $\ell$ and must pass through one of the points $p_i$. Since the tangent hyperplane at $p_i$ intersects $D_1$ in a cone with vertex at $p_i$, the line must be one of the rulings of the cone. 

Let $D_2\subset \PP^4$ be a very general hypersurface of degree $d_2 > d_1$ and let $X= D_1 \cap D_2$. By Theorem \ref{thm:main_thm}, $M_X(2, H, (d_1-1)d_2)$ has at least $d_1$ one-dimensional connected components.  
\end{example}

\bibliographystyle{plain}
  
\end{document}